\documentclass{amsart}

\usepackage[all]{xy}
\usepackage{graphicx}
\usepackage{amssymb}
\usepackage{mathrsfs}
\usepackage{multirow}
\usepackage{float}
\usepackage{tikz-cd}
\usepackage{adjustbox}
\usepackage{amsthm}

\usepackage[numbers,sort&compress]{natbib} % Ê¹²Î¿ŒÎÄÏ×Á¬Ðø±àºÅµÄÒýÓÃÒÔÑ¹Ëõ·œÊœ³öÏÖ
\usepackage[bookmarksnumbered, bookmarksopen,
colorlinks,citecolor=blue,linkcolor=blue]{hyperref}
\newcounter{RomanNumber}
\newcommand{\MyRoman}[1]{\setcounter{RomanNumber}{#1}\Roman{RomanNumber}}

\newtheorem{theorem}{Theorem}[section]
\newtheorem{lemma}[theorem]{Lemma}

\theoremstyle{definition}
\newtheorem{definition}[theorem]{Definition}

\newtheorem{proposition}[theorem]{Proposition}

\newtheorem{question}[theorem]{Question}
\theoremstyle{remark}
\newtheorem{remark}[theorem]{Remark}

\theoremstyle{notation}

\numberwithin{equation}{section}

\begin{document}

\title[Exponential growth of homotopy groups]{Exponential growth of homotopy groups of suspended finite complexes}

\author{Ruizhi Huang}
\address{Institute of Mathematics, Academy of Mathematics and Systems Science, Chinese Academy of Sciences, Beijing, China, 100190}

\email{huangrz@amss.ac.cn}
\urladdr{https://sites.google.com/site/hrzsea}

\author{Jie Wu}
\address{College of Mathematics and Information Science, Hebei Normal University, Shijiazhuang, China, 050024}

\email{wujie@hebtu.edu.cn}
\urladdr{http://www.math.nus.edu.sg/~matwujie}

\thanks{The first author is supported by Postdoctoral International Exchange Program for Incoming Postdoctoral Students under Chinese Postdoctoral Council and Chinese Postdoctoral Science Foundation.
He is also supported in part by Chinese Postdoctoral Science Foundation (Grant No. 2018M631605, and Grant No. 2019T120145), and National Natural Science Foundation of China (Grant No. 11801544). 
The second author was partially supported by National Natural Science Foundation of China (Grant No. 11028104).}

%    General info

\numberwithin{theorem}{section}
\begin{abstract}
We study the asymptotic behavior of the homotopy groups of simply connected finite $p$-local complexes, and define a space to be locally hyperbolic if its homotopy groups have exponential growth. Under certain conditions related to the functorial decomposition of loop suspension, we prove that the suspended finite complexes are locally hyperbolic if suitable but accessible information of the homotopy groups is assumed. In particular, we prove that Moore spaces are locally hyperbolic, and other candidates are also given.
\end{abstract}

\maketitle
\tableofcontents
\newpage

\section{Introduction}
\noindent In rational homotopy theory, there is a classical rational dichotomy characterizing the rational spaces of finite type (page $452$ of \cite{GTM205}):

\textit{Any simply connected space $X$ with rational homology of finite type and finite rational LS category is either:
\newline
- rationally elliptic, with $\pi_\ast(X)\otimes \mathbb{Q}$ finite dimensional, or else
\newline
- rationally hyperbolic, with $\pi_\ast(X)\otimes \mathbb{Q}$ growing exponentially.}

It is interesting to consider the corresponding questions in $p$-local homotopy theory. Unlike in rational homotopy where there are Sullivan models and Quillen models determining the rational homotopy types through purely algebraic ways, it should be much harder to develop natural or uniform methods to study the asymptotic behavior of the homotopy groups in the $p$-local setting.

Nevertheless much progress has been made in the history. For instance, Serre \cite{Serre}, Umeda \cite{Umeda} and McGibbon-Neisendorfer \cite{MN} showed that simply connected finite complexes with non-trivial modulo $p$ homology have infinitely many $p$-torsion classes in homotopy groups. Their techniques are Poincar\'{e} series, methods of analytic number theory and a theorem of Miller on contractibility of mapping spaces with source the classifying space $B\mathbb{Z}/p$. For concrete examples, the families of infinitely many higher order torsions in homotopy groups of Moore spaces were largely investigated by Cohen, Moore, Neisendorfer, Wu, Chen and others (\cite{Cohen, Cohen86, Cohen95, Chen, Neisen3}, also see the list in Section $2.2$).

In this paper, we will use homotopy decomposition techniques to study exponential growth of torsion summands in homotopy groups. In order to address to the questions, we may first introduce some definitions.

\begin{definition}\label{hyperdef}
A $p$-local complex $X$ is called $\mathbb{Z}/p^r$-\textit{hyperbolic} if the number of $\mathbb{Z}/p^r$-summands in $\pi_\ast(X)$ has exponential growth, i.e.,
\begin{equation}\label{exp}
\liminf_n\frac{{\rm ln}~ t_n}{n}>0,
\end{equation}
where $t_n=\sharp ~\{ \mathbb{Z}/p^r-{\rm summands}~{\rm in}~\oplus_{m\leq n} \pi_m(X)\}$.
\end{definition}

\begin{definition}\label{hyperpdef}
A $p$-local complex $X$ is called $p$-\textit{hyperbolic} (or \textit{hyperbolic} ${\rm mod}~p$) if the $p$ primary torsion part of $\pi_\ast(X)$ has exponential growth, i.e.,
\begin{equation}\label{exp}
\liminf_n\frac{{\rm ln}~ T_n}{n}>0,
\end{equation}
where $T_n=\sharp ~\{ \mathbb{Z}/p^r-{\rm summands}~{\rm in}~\oplus_{m\leq n} \pi_m(X), r\geq 1\}$.
\end{definition}

In either case, we may simply call $X$ \textit{locally hyperbolic} whenever there is no ambiguity. Once we have such definitions, there are some natural questions related to them. The first interesting question is that whether the limits in the definitions can be infinity or not. Actually, this was answered negatively by Henn \cite{Henn} using the unstable Adams spectral sequence.

\begin{proposition}[Corollary in \cite{Henn}, Theorem \ref{upper}]\label{mainupper}
Suppose $X$ is a simply connected finite complex. Then
the $p$-primary torsion of $\pi_\ast(X)$ has at most exponential growth, i.e.,
\begin{equation*}
\limsup_n\frac{{\rm ln}~ T_n}{n}<\infty,
\end{equation*}
where $T_n$ is as in Definition \ref{hyperpdef}.
\end{proposition}

Also from the definition, it is clear that the $\mathbb{Z}/p^r$-hyperbolicity implies the $p$-hyperbolicity, but the converse statement is not obviously true. On the contrary, if the space under consideration has finite $p$-exponent, i.e., there are no $\mathbb{Z}/p^r$ summands in the homotopy groups for sufficiently large $r$, then the $p$-hyperbolicity also implies the $\mathbb{Z}/p^r$-hyperbolicity for some $r$. For any finite simply connected $CW$-complex $X$, there is a conjecture of Moore \cite{Selick83, Anick, MW} which states that $X$ has finite $p$-exponent for any $p$ if and only if $\pi_\ast(X)\otimes \mathbb{Q}$ is a finite dimensional vector space. If this conjecture is true, it seems that there may be connections between the $p$-hyperbolicity and the usual hyperbolicity for rational homotopy. Besides, there are some spaces known to have finite $p$-exponents for odd primes $p$, such as spheres, finite $H$-spaces and ${\rm mod}~p$ Moore spaces \cite{Cohen, Neisen83, Stanley}, and for $p=2$, ${\rm mod}~2^r$ Moore spaces have finite $2$-exponent when $r\geq 2$ \cite{Theriault}. The Moore conjecture for ${\rm mod}~2$ Moore space is still open.

It will be interesting if either we can find an example that it is $p$-hyperbolic but not $\mathbb{Z}/p^r$-hyperbolic for any $r$, or we can prove that these two concepts are indeed equivalent to each other.

Here, our main concern is the following question:

\begin{question}\label{mainques}
Given any simply connected finite complex $X$ which is a suspension,

$1)$~if there are infinitely many $\mathbb{Z}/p^r$ summands in the homotopy groups $\pi_\ast(X)$, is then $X$ $\mathbb{Z}/p^r$-hyperbolic?

$2)$~if there are infinitely many nontrivial summands in the $p$-primary torsion of $\pi_\ast(X)$, is then $X$ $p$-hyperbolic?
\end{question}

In this paper, we mainly focus on the first question with some extra conditions. In \cite{Selick, Selick1}, Selick and the second author have established a functorial homotopy decomposition of the loop suspension of any path-connected CW complex by introducing the new functors $\tilde{Q}^{\rm max}_{n}$, $n\geq 2$ and $\tilde{A}^{\rm min}$:
\begin{equation*}
\Omega\Sigma X\simeq \tilde{A}^{\rm min}(X)\times \Omega\Big(   \bigvee_{n=2}^{\infty} \tilde{Q}^{\rm max}_{n}(X)\Big).
\end{equation*}
With this decomposition, we can sketch one of our main results as following:
\begin{theorem}[Theorem \ref{hyper2}, and Theorem \ref{hyperp}]\label{mainhyper}
Suppose $X$ is a path-connected finite complex localized at $p$ such that

$1)$ $\Sigma^{\ast}X$ is a homotopy retract of $\tilde{Q}^{\rm max}_{\ast}(X)$ ;

$2)$ There exists some map
\begin{equation*}
\Sigma^{\ast}X\vee \Sigma^{\ast}X\longrightarrow \Sigma X^{\wedge\ast}
\end{equation*}
which admits a left homotopy inverse;

$3)$ There is a $\mathbb{Z}/p^r$-summand in each $\pi_{\ast}(\Sigma^{\ast}X)$.

Then $\Sigma X$ is $\mathbb{Z}/p^r$-hyperbolic.

Here, the symbols $\ast$'s refer to some different arithmetic sequences which will be accurately assigned in Theorem \ref{hyper2} and Theorem \ref{hyperp} with mild arithmetic conditions.
\end{theorem}

Note that we only state the rough version here for it suffices for the reader to get the essential requirements of the result. The three conditions here are reasonable as we will see in the sequel. There are indeed some concrete important examples which satisfy the similar conditions of Theorem \ref{mainhyper} in principle.
The first examples are the local Moore spaces which have been widely studied since Cohen, Moore and Neisendorfer published their famous results on the solution to the exponent problem of spheres in 1979 \cite{Cohen, Cohen79, Cohen1}. The local Moore spaces may be divided into two families: the family of ${\rm mod}~2$ Moore spaces $P^n(2)$, and the family of ${\rm mod}~p^r$ Moore space $P^{n}(p^r)$ with $p$ odd, $r\geq 1$ and ${\rm mod}~2^r$ Moore spaces $P^n(2^r)$ with $r\geq 2$. $P^{n}(p^r)$ and $P^n(2^r)$ ($r\geq 2$) are viewed to be similar because the loop of these Moore spaces have similar decompositions \cite{Cohen, Cohen89, Neisen3} and the smash product of two such Moore spaces is homotopy equivalent to a wedge (Proposition $6.2.2$ of \cite{Neisen}). However, there are no complete answers for $P^n(2)$ to the questions of the loop decomposition and the self-smash decomposition, while part of the information has been obtained by the second author in \cite{Wu, Wu2003}. Furthermore, the $p$-primary torsions of Moore spaces, especially the summands of highest order when $p$ is odd, have also been widely investigated \cite{Cohen86, Cohen, Neisen3, Wu2003, Chen}. Based on all the mentioned information, we will see Moore spaces satisfy the condition $2)$ and $3)$ of Theorem \ref{mainhyper}, and then we have the following theorem on the hyperbolicity of Moore spaces.

\begin{theorem}[Theorem \ref{Moore2}, Theorem \ref{oddmoore}, and Theorem \ref{moore2r}]\label{hyperMoore}
For any prime $p$, $P^{n}(p^r)$ is $\mathbb{Z}/p^{r}$- and $\mathbb{Z}/p^{r+1}$-hyperbolic with $n\geq 3$ and $r\geq 1$. Also, $P^{n}(2)$ is $\mathbb{Z}/8$-hyperbolic for each $n\geq 3$.
\end{theorem}
This theorem shows the complexity of the unstable homotopy groups of Moore spaces, so that one may feel  frustrated with the computations of the homotopy groups even via computer programs. On the other hand, the same question in the stable context is also interesting. In \cite{Cohen79} it was showed that there is a homotopy commutative diagram which connects the homotopy groups of spheres with those of Moore spaces. Hence, it may be now a good place to remark that the special case of Question \ref{mainques} for spheres deserves an independent statement:
\begin{question}
Is $S^n$ ($n\geq 2$) $p$-hyperbolic?
\end{question}

We also study some other possible candidates which may be $p$-hyperbolic, based on an observation on the work of Beben and the second author \cite{Beben}. Suppose $X$ is a suspended finite $CW$-complex localized at $p$ such that either $V_{{\rm odd}}=0$ or $V_{{\rm even}}=0$ with $V=\tilde{H}_\ast(X; \mathbb{Z}/p)$. Then under some conditions, it is proved in \cite{Beben} that the suitable iterated suspension of $X$ is a homotopy retract of $L_\ast(X)$, where $L$ is the geometric realization of the algebraic Lie functor (see Section $3$), and the image of the homology of this suspension in the graded tensor algebra $T(V)$ can be well described in terms of Lie actions. Our observation is that by applying different but the same type of Lie actions, we can get many pieces of suspensions of $X$ that are all the wedge summands of $L_\ast(X)$ (Proposition \ref{multiBeben}). In particular, it means that such $X$ satisfies condition $2)$ of Theorem \ref{mainhyper}. Then in order to prove the $p$- or $\mathbb{Z}/p^r$-hyperbolicity of such spaces in the spirit of the proof of the previous results, it suffices to establish some suitable knowledge of the homotopy groups of $X$, for instance, to find a $\mathbb{Z}/p^r$ summand in each term of some arithmetic sequence $\pi_\ast(\Sigma^\ast X)$ as before.

Before the end of the introduction, it is also interesting to raise the following question:

\begin{question}\label{gap}
Suppose $X$ is a simply connected $p$-hyperbolic (or $\mathbb{Z}/p^r$-hyerbolic) finite $CW$ complex of dimension $n$. Are there numbers $\lambda>0$ and $C>1$ such that
\begin{equation*}
T_{s+a(n)}-T_{s} ~~({\rm or}~t_{s+a(n)}-t_{s})\geq \lambda C^s
\end{equation*}
for some linear function $a(n)$ and any $s\geq 1$, where $T$ and $t$ are as in Definition \ref{hyperpdef} and \ref{hyperdef}?
\end{question}
This question is a weaker $p$-local version of Question $6$ in Chapter $39$ of \cite{GTM205}, where F\'{e}lix, Halperin and Thomas name a positive answer to the rational version of the problem when $\lambda C^s=1$ by the gap theorem. Hence we may say a positive answer to Question \ref{gap} is a $p$- (or $\mathbb{Z}/p^r$-) hyperbolic gap theorem for $X$. Indeed, for the rational case F\'{e}lix Halperin and Thomas have even showed a stronger result in \cite{Felix}, where they also raise a question similar to Question \ref{gap} in the local setting.

The paper is organized as follows. In section $2$, we consider the hyperbolicity of wedges of Moore spaces by the Hilton-Milnor theorem. We then prove the hyperbolicity of Moore spaces (Theorem \ref{hyperMoore}) in section $3$ and $4$. Theorem \ref{mainhyper} is proved in Section $5$ as a connection to functorial homotopy decomposition. Section $6$ is devoted to other possible candidates. We also add an appendix to reprove the fact that the homotopy groups have at most exponential growth.

\section{The hyperbolicity of wedges of Moore spaces}
\subsection{$\mathbb{Z}/p^r$-hyperbolicity} First, let us recall the Hilton-Milnor theorem (\cite{White}, Chapter \MyRoman{11}, Section $6$) which describes certain homotopy decomposition of loop suspension of wedges. Let $X_1, \ldots, X_k$ be connected CW-complexes with base points. Then there is a (weak) homotopy equivalence
\begin{equation}\label{Hilton-Milnor}
\Omega\Sigma (X_1\vee X_2\vee \cdots\vee X_k) \simeq \prod_{\alpha\in \mathcal{I}} \Omega \Sigma (X_1^{\wedge a_1}\wedge X_2^{\wedge a_2}\wedge \cdots \wedge X_k^{\wedge a_k}),
\end{equation}
where $\mathcal{I}$ is a vector space basis for a free Lie algebra $\mathbb{L}[y_1, y_1,\ldots ,y_{k}]$ on $k$ generators $y_1, y_2, \ldots y_k$, and $a_1,a_2, \ldots, a_k$ are non-negative integers which record the number of occurrences of $y_1, y_2, \ldots y_k$ respectively in the bracket $\alpha$. We call $n=a_1+a_2+\cdots +a_k$ the weight of $\alpha$.
Note by convention $
X_1^{\wedge a_1}\wedge X_2^{\wedge a_2}\wedge \cdots \wedge X_k^{\wedge a_k}=X_1$ if $\alpha=y_1$ and similar conventions are assumed for other generators.
Furthermore, the dimension of the subspace consisting of weight $n$ elements in $\mathbb{L}[y_1, y_1,\ldots ,y_{k}]$ is determined by the Witt formulae

\begin{equation*}
W(n)=\frac{1}{n} \sum\limits_{d|n}\mu(d)k^{n/d},
\end{equation*}
where $\mu$ in the M\"{o}bius function defined by $\mu(1)=1$, $\mu(n)=0$ unless $n>1$ is square free, and $\mu(p_1\cdots p_l)=(-1)^l$ if $p_1, \ldots, p_l$ are distinct primes.

In order to study hyperbolicity of wedges of Moore spaces, we also need the bouquet splitting of smash product of Moore spaces. Let $p$ be a prime and $m$, $n\geq 2$. If with $p^r\neq 2$, then there is a homotopy equivalence (Proposition $6.2.2$ of \cite{Neisen})
\begin{equation}\label{qnot2}
P^{m}(p^r)\wedge P^{n}(p^r)\simeq P^{m+n}(p^r)\vee P^{m+n-1}(p^r).
\end{equation}
By contrast, there is no bouquet splitting of $2$-fold smash product of ${\rm mod}~2$ Moore spaces. However, the second author (Corollary $3.7$ of \cite{Wu2003}) proved that the $3$-fold smash product of real projective spaces splits
\begin{equation}\label{qis2}
(\mathbb{R}P^2)^{\wedge 3}\simeq \mathbb{R}P^2\wedge \mathbb{C}P^2\vee P^5(2)\vee P^5(2).
\end{equation}

\begin{proposition}\label{qhyperwedge}
$P^{m+1}(p^r)\vee P^{n+1}(p^r)$ is $\mathbb{Z}/p^r$-hyperbolic for each $m$, $n\geq 2$.
\end{proposition}
\begin{proof}
First, by the Hilton-Milnor theorem we have
\begin{eqnarray*}
 &  &\Omega(P^{m+1}(p^r)\vee P^{n+1}(p^r))\\
&\simeq& \prod_{\alpha\in \mathbb{L}[y_1,y_2]} \Omega\Sigma (P^m(p^r)^{\wedge a_1}\wedge P^n(p^r)^{\wedge a_2})\\
&\hookleftarrow& \prod_{{\rm weight}= 2s+1} \Omega\Sigma (P^m(p^r)^{\wedge a_1}\wedge P^n(p^r)^{\wedge a_2}).\\
\end{eqnarray*}
When $p^r\neq 2$, each $P^m(p^r)^{\wedge a_1}\wedge P^n(p^r)^{\wedge a_2}$ splits into a bouquet of $2^{2s}$ Moore spaces by applying (\ref{qnot2}) repeatedly; similarly when $p^r=2$, it splits into a bouquet of $2^s$ ${\rm mod}~2$ Moore spaces by applying (\ref{qis2}). Hence in either case, we have an inclusion of homotopy groups
\begin{equation*}
\pi_{\ast+1}(P^{m+1}(p^r)\vee P^{n+1}(p^r))\hookleftarrow \bigoplus_{W(2s+1)\cdot 2^s}\pi_{\ast+1}(P(p^r)).
\end{equation*}
Here, $P(p^r)$'s represent Moore spaces of different dimensions but at most $(2s+1)m+1$ (we may suppose $m\geq n$).
Since there is a $\mathbb{Z}/p^r$ summand in each $\pi_{\ast+1}(P(p^r))$ by Hurewicz theorem, we have
\begin{equation*}
\liminf_n\frac{{\rm ln}~ t_n}{n}=
\liminf_s\frac{{\rm ln}~ t_{(2s+1)m}}{(2s+1)m}
\geq
\liminf_s\frac{{\rm ln}~ (W(2s+1) \cdot 2^s)}{(2s+1)m}
>0,
\end{equation*}
where $t_n=\sharp ~\{ \mathbb{Z}/p^r-{\rm summands}~{\rm in}~\oplus_{m\leq n} \pi_m(P^{m+1}(p^r)\vee P^{n+1}(p^r))\}$, and the first equality is due to the fact that $t_n$ is an increasing sequence.
\end{proof}

\subsection{$\mathbb{Z}/p^{r+1}$-hyperbolicity} For other torsion summands, we need more information about homotopy groups of Moore spaces, and there are three cases:
\begin{itemize}
\item{${\rm mod}~2$ Moore spaces.}
    \begin{itemize}
     \item \[\pi_{n+2}(P^{n+1}(2))\cong \mathbb{Z}/4\] for any $n\geq 2$ (see the table of the homotopy groups of ${\rm mod}~2$ Moore spaces in Appendix $A$ of \cite{Wu2003}).
     \item By Corollary $1.4$ of \cite{Chen}, there exists a $\mathbb{Z}/8$-summand in
     \[\pi_{K^{(i)}m+k^{(i)}}(\Sigma^{4m+i}P^{n+1}(2))\] for each $m$ and $i$ with $0\leq i\leq 3$ and some constants $K^{(i)}\in \mathbb{Z}^+$, $k^{(i)}\in \mathbb{Z}$.
    \end{itemize}
\item{${\rm mod}~p^r$ Moore spaces with $p$ an odd prime and $r\geq 1$.}

By Theorem $1.4$ of \cite{Cohen} which was later strengthened in \cite{Neisen3}, there are $\mathbb{Z}/p^{r+1}$-summands for any $m\geq 2$ in
\[\pi_{2p(m-1)-1}(P^{2m-1}(p^r)) ~{\rm and}~ \pi_{4pm-2p-1}(P^{2m}(p^r)).\]
\item{${\rm mod}~2^r$ Moore spaces with $r\geq 2$ (Lemma \ref{pi2}).} 

For any sufficiently large $n$, there exist $\mathbb{Z}/2^{r+1}$-summands in the homotopy groups
\[\pi_{4n-2}{P^{2n}(2^r)} ~{\rm and}~\pi_{12n-2}(P^{2n+1}(2^r)).\]
\end{itemize}

\begin{proposition}\label{qhyperwedge2}
For each $m$, $n\geq 2$,
\begin{itemize}
\item[1)] $P^{m+1}(p^r)\vee P^{n+1}(p^r)$ is $\mathbb{Z}/p^{r+1}$-hyperbolic.
\item[2)] $P^{m+1}(2)\vee P^{n+1}(2)$ is $\mathbb{Z}/8$-hyperbolic.
\end{itemize}
\end{proposition}
\begin{proof}
With the information of homotopy groups in hand, we can prove the proposition by using similar argument as that in the proof of Proposition \ref{qhyperwedge}.
\end{proof}
\begin{remark}
From the discussion in this section, we see that we can prove more results on local hyperbolicity if we have more suitable information about other torsion summands of homotopy groups of Moore spaces.
\end{remark}

\section{The hyperbolicity of ${\rm mod}~2$ Moore space}
\noindent We start with a type of decomposition of loop suspension developed by the second author.
Let $X$ be a path connected  $p$-local $CW$ complex of finite type and $S_k$ denote the symmetric group on $k$ letters. The natural action of any element $\delta\in \mathbb{Z}/p[S_k]$ on $V^{\otimes k}$ with $V=\tilde{H}_\ast(X;\mathbb{Z}/p)$ by permuting coordinates can be realized as a local map
\begin{equation*}
\tilde{\delta}: \Sigma X^{\wedge k}\rightarrow \Sigma X^{\wedge k}.
\end{equation*}
In particular, $\tilde{\delta}$ is represented by a corresponding element in $ \mathbb{Z}_{(p)}[S_k]$. 
Also there are the so-called \textit{Dynkin-Specht-Wever elements} $\beta_k$ which are defined by $\beta_2=1-(1,2)$ and the recursion relation
\begin{equation*}
\beta_k=(1\otimes \beta_{k-1})\circ(1-(1,2,\ldots,k)).
\end{equation*}
The action of $\beta_k$ on $V^{\otimes k}$ is then given by sending a tensor $x_1\otimes x_2\ldots \otimes x_k\in V^{\otimes k}$ to the commutator
\begin{equation*}
[x_1,[x_2,\ldots,[x_{k-1},x_k]\ldots]]\in V^{\otimes k}.
\end{equation*}
An important and useful property of $\beta_k$ is that $\beta_k\circ \beta_k=k\beta_k$. Hence if ${\rm g.c.d.}~(p,k)=1$ then
\begin{equation*}
\frac{1}{k}\beta_k\circ \frac{1}{k}\beta_k=\frac{1}{k}\beta_k
\end{equation*}
which is an idempotent. Then ${\rm id}-\frac{1}{k}\beta_k$ is also an idempotent. From their realizations $\frac{1}{k}\tilde{\beta}_k$ and ${\rm id}-\frac{1}{k}\tilde{\beta}_k$, we can define two homotopy telescopes: one of which we may denote by
\begin{equation*}
\tilde{L}_k(X)={\rm hocolim}_{\frac{1}{k}\tilde{\beta}_k} \Sigma X^{\wedge k},
\end{equation*}
and the other one by $Y$. Then it is easy to see there is homotopy equivalence
\begin{equation*}
\Sigma X^{\wedge k}\simeq \tilde{L}_k(X)\vee Y.
\end{equation*}
Furthermore, if $X$ itself is a suspension, then $\tilde{L}_k(X)\simeq \Sigma L_k(X)$ where
\begin{equation*}
L_k(X)={\rm hocolim}_{\frac{1}{k}\tilde{\beta}_k} X^{\wedge k}.
\end{equation*}
By the construction, the ${\rm mod}~p$ homology of $L_k(X)$ is the image of $\beta_k: V^{\otimes k}\rightarrow V^{\otimes k}$, which is the submodule of length $k$ Lie brackets in $V^{\otimes k}$ and usually denoted by $L_k(V)$.

\begin{theorem}[\cite{Wu}, Theorem $1.6$]\label{Liedecom}
Let $X$ be a path-connected $p$-local complex of finite type and $1<k_1<k_2\cdots$ be a sequence of integers such that

$1)$ $k_j\not\equiv 0~{\rm mod}~p$ for each $j\geq 1$;

$2)$ $k_j$ is not a multiple of any $k_i$ for each $i\neq j$.

Then
\begin{equation*}
\Omega\Sigma X\simeq \prod_i \Omega(\tilde{L}_{k_j}(X))\times ({\rm some}~ {\rm other}~{\rm space}).
\end{equation*} \hfill $\Box$
\end{theorem}
We use Theorem \ref{Liedecom} to deduce the following special decomposition of ${\rm mod}~2$ Moore space.
\begin{lemma}\label{mod2factors}
Let $n\geq 2$. There is a homotopy decomposition at prime $2$
\begin{equation*}
\Omega P^{n+1}(2) \simeq \prod_{p ~{\rm odd}}  \Omega\Sigma^{1+\frac{p^2-1}{2}(2n-1)}(P^{n}(2)\vee P^{n} (2))\times ({\rm some}~ {\rm other}~{\rm space}).
\end{equation*}
\end{lemma}
\begin{proof}
By Proposition $3.1$ of \cite{Chen} we know that $\Sigma^{1+k(2n-1)}P^{n}(2)$ is a homotopy retract of $\tilde{L}_{2k+1}(P^{n}(2))\simeq \Sigma L_{2k+1}(P^{n}(2))$, which is further a homotopy retract of $\Sigma P^n(2)^{\wedge 2k+1}$. In particular, $\Sigma^{1+\frac{p^2-1}{2}(2n-1)}P^{n}(2)$ is a homotopy retract of $\Sigma L_{p^2}(P^{n}(2))$.

On the other hand, by Theorem $1.2$ of \cite{Selick01} there is a complete decomposition over the Steenrod algebra
\begin{equation*}
\Sigma P^n(2)^{\wedge m} \simeq \bigvee_{m/2<a<m}\bigvee^{c_{(a, m-a)}} Q^{(a, m-a)}(P^{n}(2)),
\end{equation*}
in the sense that for each factor $Q^{(a, m-a)}(P^{n}(2))$ its cohomology $H^\ast(Q^{(a, m-a)}(P^{n}(2));\mathbb{Z}/p)$ is indecomposable as module of the Steenrod algebra. Moreover, either 
\[
{\rm dim}~\bar{H}_\ast(Q^{(a, m-a)}(P^{n}(2)))\equiv 0~{\rm mod}~4 ~~{\rm or}~~{\rm dim}~\bar{H}_\ast(Q^{(a, m-a)}(P^{n}(2)))=2,
\] and $Q^{(a, m-a)}(P^{n}(2))\simeq\Sigma^t(P^{n}(2))$ for some particular $t$ in the latter case.
Since $\Sigma L_{p^2}(P^{n}(2))$ is a homotopy retract of $\Sigma P^n(2)^{\wedge p^2}$, there is a homotopy decomposition
\begin{equation}\label{lemmaequL}
\Sigma L_{p^2}(P^{n}(2)) \simeq \bigvee_{{\rm some}~ a}\bigvee^{c_{(a, p^2-a)}} Q^{(a, p^2-a)}(P^{n}(2)).
\end{equation}
Then since by the Witt formulae
\begin{equation*}
{\rm dim}~(L_{p^2}(V))=\frac{1}{p^2}(2^{p^2}-2^p)\equiv 0~{\rm mod}~4,
\end{equation*}
where $V=\bar{H}_\ast(P^{n}(2); \mathbb{Z}/2)$, there must be even number of $Q^{(a, p^2-a)}(P^{n}(2))\simeq\Sigma^{1+\frac{p^2-1}{2}(2n-1)}P^{n}(2)$ in (\ref{lemmaequL}). In particular, there exists a map
\begin{equation*}
\tilde{L}_{p^2}(P^{n}(2))\hookleftarrow \Sigma^{1+\frac{p^2-1}{2}(2n-1)}P^{n}(2)\vee\Sigma^{1+\frac{p^2-1}{2}(2n-1)}P^{n}(2)
\end{equation*}
which admits a left homotopy inverse. Now from Theorem \ref{Liedecom}, we get
\begin{equation*}
\Omega P^{n+1}(2) \simeq \prod_{p~{\rm odd}} \Omega(\tilde{L}_{p^2}(P^{n}(2)))\times ({\rm some}~ {\rm other}~{\rm space})
\end{equation*}
which implies
\begin{equation*}
\Omega P^{n+1}(2)\hookleftarrow \prod_{p~{\rm odd}}  \Omega\big(\Sigma^{1+\frac{p^2-1}{2}(2n-1)}P^{n}(2)\vee\Sigma^{1+\frac{p^2-1}{2}(2n-1)}P^{n}(2)\big)
\end{equation*}
admits a left homotopy inverse.
\end{proof}

We are now in a position to prove the hyperbolicity of ${\rm mod}~2$ Moore space.
\begin{theorem}\label{Moore2}
$P^{n+1}(2)$ is $\mathbb{Z}/2^i$-hyperbolic for each $n\geq 2$ and $i=1$, $2$, $3$.
\end{theorem}
\begin{proof}
By Lemma \ref{mod2factors}, we have
\begin{eqnarray*}
\Omega P^{n+1}(2)
&\hookleftarrow& \Omega\big(\Sigma^{1+\frac{p^2-1}{2}(2n-1)}P^{n}(2)\vee\Sigma^{1+\frac{p^2-1}{2}(2n-1)}P^{n}(2)\big)
\end{eqnarray*}
which admits a left homotopy inverse. Then since the wedge of ${\rm mod}~2$ Moore spaces are $\mathbb{Z}/2^i$-hyperbolic by Proposition \ref{qhyperwedge} and Proposition \ref{qhyperwedge2}, the same holds for $P^{n+1}(2)$.
\end{proof}

\section{The hyperbolicity of ${\rm mod}~p^r$ Moore spaces}
\noindent In this section, we continue to prove the hyperbolicity of other Moore spaces.
\subsection{$p=$ odd prime}
In 1979 and 1980s, Cohen, Moore and Neisendorfer investigated the homotopy decompositions of loops of ${\rm mod}~p^r$ Moore spaces, with which they proved deep results on the homotopy exponents of Moore spaces and odd dimensional spheres. Here we use their decompositions to study local hyperbolicity.

\begin{theorem}\label{oddmoore}
$P^{n}(p^r)$ is $\mathbb{Z}/p^{r}$- and $\mathbb{Z}/p^{r+1}$-hyperbolic for each odd prime $p$, $n\geq 3$ and $r\geq 1$.
\end{theorem}
\begin{proof}
Since by Theorem $1.1$ of \cite{Cohen} there is a homotopy decomposition
\begin{equation*}
 \Omega P^{2n+2}(p^r)\simeq S^{2n+1}\{p^r\}\times \Omega \bigvee_{m=0}^{\infty}P^{4n+2mn+3}(p^r)
 \end{equation*}
for $n\geq 1$ where $S^{2n+1}\{p^r\}$ is the homotopy theoretic fibre of the degree map $p^r: S^n\rightarrow S^n$, it suffices to prove the theorem for $P^{2n+1}(p^r)$ with $n\geq 2$. In \cite{Cohen1, Neisen3}, Cohen, Moore and Neisendorfer have proved a homotopy decomposition of odd dimensional Moore spaces
\begin{equation*}
\Omega P^{2n+1}(p^r)\simeq \Omega \Sigma \bigvee_\alpha P^{n_\alpha}(p^r)\times T^{2n+1}\{p^r\},
\end{equation*}
where $T^{2n+1}\{p^r\}$ is the atomic piece of $\Omega P^{2n+1}(p^r)$ containing the bottom cell and $n_\alpha$ runs over an index set of integers greater than $4n-1$ and there are only finitely many $n_\alpha$ with a given value. Our theorem then follows from Proposition \ref{qhyperwedge} and Proposition \ref{qhyperwedge2}.
\end{proof}

\subsection{$p=2$ and $r\geq 2$}
First, let us prove the following lemma to complete the proof of Proposition \ref{qhyperwedge2}.
\begin{lemma}\label{pi2}
There exist $\mathbb{Z}/2^{r+1}$-summands in the homotopy groups $\pi_{4n-2}{P^{2n}(2^r)}$ and $\pi_{12n-2}(P^{2n+1}(2^r))$ for any sufficiently large $n$.
\end{lemma}
\begin{proof}
In Section $21$ of \cite{Cohen86}, Cohen has constructed an element in $\pi_{4n-3}(\Omega P^{2n}(2^r))$ which corresponds to the Lie element $[\mu, \nu]\in H_{4n-3}(\Omega P^{2n}(2^r); \mathbb{Z}/2^r)$ for any $n\geq 2$, and also proved that it is of order exactly $2^{r+1}$ if $r\geq 2$ and the Whitehead product $\omega_{2n-1}$ is not divisible by $2$. Then according to the recent famous solution to the Kervaire invariant one problem by Hill, Hopkins and Ravenel \cite{Hill}, we know that $\omega_{2n-1}$ is not divisible by $2$ for any $n>64$. Hence, there exists a $\mathbb{Z}/2^{r+1}$-summand in $\pi_{4n-2}{P^{2n}(2^r)}$ for $n>64$.

For the odd dimension case, there is a homotopy equivalence
\begin{equation*}
P^{6n}(2^r)\simeq \Sigma^{1+(4n-1)}P^{2n}(2^r)\simeq \tilde{L}_3(P^{2n}(2^r))
\end{equation*}
by Proposition $3.1$ of \cite{Chen}. Hence the composition map
\begin{equation*}
\Omega P^{6n}(2^r) \simeq \Omega \tilde{L}_3(P^{2n}(2^r)) \hookrightarrow \Omega P^{2n+1}(2^r)
\end{equation*}
admits a left homotopy inverse by Theorem \ref{Liedecom}, and then the lemma follows from previous discussion.
\end{proof}
\begin{theorem}\label{moore2r}
$P^{n}(2^r)$ is $\mathbb{Z}/2^{r}$- and $\mathbb{Z}/2^{r+1}$-hyperbolic for each $n\geq 3$ and $r\geq 2$.
\end{theorem}
\begin{proof}
In this case, there is also a homotopy decomposition \cite{Cohen89}
\begin{equation*}
\Omega P^m(2^r)\simeq T^{m}(2^r)\times \Omega \bigvee_{n\in \mathcal{I}} P^{n}(2^r)
\end{equation*}
for each $m\geq 3$ and $r\geq 2$. The theorem again follows from Proposition \ref{qhyperwedge} and Proposition \ref{qhyperwedge2}.
\end{proof}
\begin{remark}
The readers may notice that we did not invoke Kervaire invariant one problem for the proof of hyperbolicity of odd primary Moore spaces, while it seems to be necessary for that of ${\rm mod}~2^r$ Moore spaces. The explanation is as follows. In order to get higher order torsions in the integral homotopy groups of odd primary Moore spaces, Cohen-Moore-Neisendorfer \cite{Cohen} actually studied homotopy groups with coefficients introduced by Frank Peterson. In this context, Samelson products are defined through the bouquet splitting of suspensions of Moore spaces (\ref{qnot2}).
As pointed by Neisendorfer \cite{Neisen2013}, unlike the odd prime case the ${\rm mod}~ 2^r$ homotopy groups ($r\geq 2$) is not a graded Lie algebra under the Samelson product (further the squaring operations probably can not be defined). The key part of proof of Cohen-Moore-Neisendorfer is to use the induced graded Lie algebra structures on ${\rm mod}~ p$ homotopy Bockstein spectral sequences which we do not have in the ${\rm mod}~ 2$ case. Then one can not apply the method of Cohen-Moore-Neisendorfer to obtain higher order torsions. Fortunately, Cohen \cite{Cohen86} reduced a particular case of the problem to the indivisibility of Whitehead products by classical homotopy techniques, which were deeply investigated by Hill, Hopkins and Ravenel \cite{Hill}.
\end{remark}

\section{Functorial decomposition and hyperbolicity}
\noindent  By observing the discussions in the previous sections, we see there are four important ingredients in our proof of local hyperbolicity of a suspended complex $X$:
\begin{itemize}
\item[0)] The Hilton-Milnor Theorem.
\item[1)] Homotopy decomposition of $\Omega X$ with loop of a bouquet of suspensions of $X$ as one of the factors.
\item[2)] Bouquet splitting of smash product of suspensions of $X$.
\item[3)] Suitable knowledge of homotopy groups of suspensions of $X$.
\end{itemize}

Hence it is possible to generalize our concrete results to some abstract theorems. In this section, we study the hyperbolicity with the help of functorial decomposition, for which let us recall a powerful functorial decomposition of loop suspension developed by Selick and the second author.

We start with the algebraic side. Let $V$ be a vector space over a field (e.g., $\mathbb{Z}/p$) and $T(V)$ be the tensor algebra generated by $V$. $T(V)$ becomes a Hopf algebra by letting $V$ be primitive. Consider $T$ as a functor from modules to coalgebras. Selick and the second author (Theorem $1.4$ of \cite{Selick}) showed that there exists a functor $A^{\rm min}$ as a smallest retract of $T$ in the following sense
\begin{itemize}
\item[1)] $A^{\rm min}(V)$ is a functorial coalgebra retract of $T(V)$.
\item[2)] $V\subseteq A^{\rm min}(V)$.
\item[3)] If $A(V)$ is any functorial coalgebra retract of $T(V)$ with $V\subseteq A(V)$, then $A^{\rm min}(V)$ is a functorial coalgebra retract of $A(V)$.
\end{itemize}
Furthermore, $A^{\rm min}$ is unique up to natural equivalence.
\begin{theorem}[\cite{Selick}, Theorem $6.5$, Corollary $8.9$]
There is a functorial coalgebra decomposition
\begin{equation*}
T(V) \cong  A^{\rm min}(V)\otimes T(\bigoplus_{n=2}^{\infty}Q^{\rm max}_{n}(V)),
\end{equation*}
where $Q^{\rm max}_{n}(V)$ is a functorial retract of $V^{\otimes n}$ and a sub-functor of $L_n(V)$, and $T(\bigoplus_{n=2}^{\infty}Q^{\rm max}_{n}(V))$ is a sub Hopf algebra of $T(V)$. \hfill $\Box$
\end{theorem}
These algebraic results admit geometric realizations.
\begin{theorem}[\cite{Selick1}, Theorem $1.2$]\label{Decom}
Let $X$ be any path-connected $p$-local CW complex. Then there are homotopy functors $\tilde{Q}^{\rm max}_{n}$, $n\geq 2$ and $\tilde{A}^{\rm min}$ from path-connected $p$-local CW complexes to spaces with the following properties

$1)$  $\tilde{Q}^{\rm max}_{n}$ is a functorial retract of $\Sigma X^{\wedge n}$ and there exists a map $f$ such that
\begin{equation*}
\tilde{Q}^{\rm max}_{n}(X)\simeq {\rm hocolim}_{f} ~\Sigma X^{\wedge n}
\end{equation*}

$2)$ there is a functorial decomposition
\begin{equation*}
\Omega\Sigma X\simeq \tilde{A}^{\rm min}(X)\times \Omega\Big(   \bigvee_{n=2}^{\infty} \tilde{Q}^{\rm max}_{n}(X)\Big).
\end{equation*} \hfill $\Box$
\end{theorem}
Now we can prove our main theorems of this section, the ${\rm mod}~2$ version of which is as follows.
\begin{theorem}\label{hyper2}
Suppose $X$ is a path-connected finite complex localized at $2$ such that

$1)$ $\Sigma^{nM+1}X$ is a homotopy retract of $\tilde{Q}^{\rm max}_{nl+1}(X)$ for some fixed $M$, $l\in \mathbb{Z}^+$ and sufficiently large $n$;

$2)$ There exists some $n$ satisfying $1)$ such that the multiplicity of $\Sigma^{nM+1}X$ in the decomposition of $\Sigma X^{\wedge(nl+1)}$ is bigger than $1$, i.e., there exists some map
\begin{equation*}
\Sigma^{nM+1}X\vee \Sigma^{nM+1}X\longrightarrow \Sigma X^{\wedge(nl+1)}
\end{equation*}
which admits a left homotopy inverse;

$3)$ There is a $\mathbb{Z}/2^r$-summand in $\pi_{Kn+k}(\Sigma^{An+a}X)$ for sufficiently large $n$ and some fixed $K$, $A\in\mathbb{Z}^+$ and $k$, $a\in \mathbb{Z}$.

Then $\Sigma X$ is $\mathbb{Z}/2^r$-hyperbolic if ${\rm g.c.d.}(M+li,A)=1$ for any integer $i$ with $0\leq i\leq A-1$.
\end{theorem}

\begin{proof}
We prove this theorem by four steps.

\textit{\textbf{Step $1$}: prove that there exists a $\mathbb{Z}/2^r$-summand in $\pi_{K^{(i)}n+k^{(i)}}(\Sigma^{An+i}X)$ for each integer $i$ with $0\leq i\leq A-1$ and sufficiently large $n$, where $K^{(i)}\in \mathbb{Z}^+$, $k^{(i)}\in \mathbb{Z}$ are some constants.}

By Lemma $2.2$ of \cite{Selick1} and Lemma $2.1$ of \cite{Selick}, $\tilde{Q}^{\rm max}_{nl+1}(X)$ is a functorial homotopy retract of $\Sigma X^{\wedge(nl+1)}$ such that
\begin{equation*}
\tilde{Q}^{\rm max}_{nl+1}(X)\simeq {\rm hocolim}_{f} ~\Sigma X^{\wedge(nl+1)}
\end{equation*}
for some $f=\sum\limits_{\sigma\in S_{nl+1}} k_\sigma \sigma$. When $X$ is sphere, the map $\sigma$ induces identity on ${\rm mod}~2$ homology. The functoriality
of $\tilde{Q}^{\rm max}$ then implies
the composition map
\begin{equation*}
\Sigma^{nl+1} \tilde{Q}^{\rm max}_{nl+1}(X)\rightarrow \Sigma^{nl+2}X^{\wedge nl+1}\simeq \Sigma(\Sigma X)^{\wedge nl+1}\rightarrow
\tilde{Q}^{\rm max}_{nl+1}(\Sigma X)
\end{equation*}
is a homotopy equivalence.
More generally, we have
\begin{equation*}
\tilde{Q}^{\rm max}_{nl+1}(\Sigma^t X)\simeq \Sigma^{(nl+1)t} \tilde{Q}^{\rm max}_{nl+1}(X).
\end{equation*}
Then by condition $1)$, we have $\Sigma^{nM+1+(nl+1)t}X$ is a homotopy retract of $\tilde{Q}^{\rm max}_{nl+1}(\Sigma^t X)$ for sufficiently large $n$. Then we see that the loop map
\begin{equation*}
\Omega \Sigma^{nM+1+(nl+1)t}X\rightarrow \Omega \tilde{Q}^{\rm max}_{nl+1}(\Sigma^t X) \rightarrow \Omega \Sigma^{t+1} X
\end{equation*}
has a left homotopy inverse. Now let $t=An^\prime +i$, we see that $nM+1+(nl+1)t=A(nl+1)n^\prime+(M+li)n+i+1$. If
\begin{equation}\label{equiv}
A(nl+1)n^\prime+(M+li)n+i+1\equiv a ~{\rm mod}~ A
\end{equation}
holds, then by condition $3)$ there is a $\mathbb{Z}/2^r$-summand in $\pi_\ast(\Omega \Sigma^{A(nl+1)n^\prime+(M+li)n+i+1}X)$, and then in $\pi_\ast(\Omega\Sigma^{An^\prime+i+1}X)$ for sufficiently large $n$. Since by hypothesis ${\rm g.c.d.}(M+li,A)=1$ for any integer $i$ with $0\leq i\leq A-1$, then the congruent equation (\ref{equiv}) has a solution for each $i$ with $0\leq i\leq A-1$. For each $i$, choose a sufficiently large $n$ satisfying (\ref{equiv}), then there is a $\mathbb{Z}/2^r$-summand in $\pi_\ast(\Omega\Sigma^{An^\prime+i+1}X)$ and we see the claim of Step $1$ holds.

\textit{\textbf{Step $2$}: prove that the multiplicity of $\Sigma^{snM+1}X$ in the decomposition of $\Sigma X^{\wedge(snl+1)}$ is at least $2^s$.}

This step can be done easily by using condition $2)$.

\textit{\textbf{Step $3$}: prove that $\Sigma^\lambda X \vee \Sigma^\lambda X$ is $\mathbb{Z}/2^r$-hyperbolic for any $\lambda>0$.}

With the results in Step $1$ and $2$, Step $3$ can be showed by similar argument as that in the proof of Proposition \ref{qhyperwedge}.

\textit{\textbf{Step $4$}: prove the theorem.}

By Theorem \ref{Decom}, we have
\begin{equation*}
\Omega\Sigma X\simeq  \Omega\Big(   \bigvee_{n=2}^{\infty} \tilde{Q}^{\rm max}_{n}(X)\Big)\times\tilde{A}^{\rm min}(X).
\end{equation*}
Also, condition $1)$ implies
\begin{equation*}
\Omega\Big( \bigvee_{{\rm large}~n} \Sigma^{nM+1} X\Big)\hookrightarrow \Omega\Big(   \bigvee_{n=2}^{\infty} \tilde{Q}^{\rm max}_{nl+1}(X)\Big)
\end{equation*}
has a left homotopy inverse. Hence, $\Omega\Big( \bigvee_{{\rm large}~n} \Sigma^{nM+1} X\Big)$ is a homotopy retract of
$\Omega\Sigma X$. Fix some $n$ satisfying condition $1)$ and $2)$. We see in particular
\begin{equation*}
\Omega\Sigma X\hookleftarrow \Omega\Big(\Sigma^{nM+1} X\vee \Sigma^{(n+1)M+1} X \vee\cdots \vee \Sigma^{(n+nl)M+1} X \Big)
\end{equation*}
admits a left homotopy inverse.
Then by the Hilton-Milnor theorem, we have
\begin{equation*}
\Omega\Sigma\Big(\Sigma^{nM} X\wedge \Sigma^{(n+1)M} X \wedge\cdots \wedge \Sigma^{(n+nl)M} X \Big)
\simeq \Omega\Sigma\Big( \Sigma^{(nl+1)(n+1/2 nl)M}X^{\wedge (nl+1)}\Big)
\end{equation*}
is a factor of $\Omega\Sigma X$ which corresponds to the term $[y_0, y_1,\ldots ,y_{nl}]$ in the free Lie algebra $\mathbb{L}[y_0, y_1,\ldots ,y_{nl}]$. Denote $D_n=(nl+1)(n+1/2 nl)$, then by condition $2)$ we have
\begin{equation*}
\Omega\Sigma\big( \Sigma^{D_nM}X^{\wedge (nl+1)}\big)\hookleftarrow \Omega\Sigma^{D_nM}(\Sigma^{nM+1}X\vee \Sigma^{nM+1}X)
\end{equation*}
which admits a left homotopy inverse for some $n$. The theorem then follows from Step $3$.
\end{proof}

\begin{theorem}\label{hyperp}
Suppose $X$ is a path-connected finite complex localized at an odd prime $p$ such that

$1)$ $\Sigma^{2nM+1}X$ is a homotopy retract of $\tilde{Q}^{\rm max}_{nl+1}(X)$ for some fixed $M$, $l\in \mathbb{Z}^+$ and sufficiently large $n$;

$2)$ There exists some $n$ satisfying $1)$ such that the multiplicity of $\Sigma^{2nM+1}X$ in the decomposition of $\Sigma X^{\wedge(nl+1)}$ is bigger than $1$, i.e., there exists some map
\begin{equation*}
\Sigma^{2nM+1}X\vee \Sigma^{2nM+1}X\longrightarrow \Sigma X^{\wedge(nl+1)}
\end{equation*}
which admits a left homotopy inverse;

$3)$ There is a $\mathbb{Z}/p^r$-summand in $\pi_{Kn+k}(\Sigma^{2An+2a+1}X)$ for sufficiently large $n$ and some fixed $K$, $A\in\mathbb{Z}^+$ and $k$, $a\in \mathbb{Z}$.

Then $\Sigma X$ is $\mathbb{Z}/p^r$-hyperbolic if ${\rm g.c.d.}(M+li, A)=1$ for any integer $i$ with $0\leq i\leq A-1$.
\end{theorem}

\begin{proof}
The proof of this theorem is the same as that of Theorem \ref{hyper2} except minor modifications in Step $1$.
That is, in this ${\rm mod}~p$ case $\sigma$ induces ${\rm sgn}~(\sigma)$ on ${\rm mod}~p$ homology when $X$ is sphere, and we only have
\begin{equation*}
\tilde{Q}^{\rm max}_{nl+1}(\Sigma^2 X)\simeq \Sigma^{2(nl+1)} \tilde{Q}^{\rm max}_{nl+1}(X).
\end{equation*}
Similarly,
\begin{equation*}
\tilde{Q}^{\rm max}_{nl+1}(\Sigma^{2t} X)\simeq \Sigma^{2(nl+1)t} \tilde{Q}^{\rm max}_{nl+1}(X).
\end{equation*}
These modifications then results in some numerical changes in the statement of the theorem.
\end{proof}
\begin{remark}
We note that Moore spaces satisfy the conditions of Theorem \ref{hyper2} or Theorem \ref{hyperp} except the first condition involving $\tilde{Q}^{\rm max}$. But there are some evidences which suggest that condition $1)$ may be valid for Moore spaces.

Consider the hyperbolicity of $P^{n+1}(2)$. In Section $4.1$ of \cite{Wu2003}, the second author has proved that there is a Moore space as a summand of $\tilde{Q}^{\rm max}_i(P^{n+1}(2))$ for each $i=3$, $5$, $7$, $9$. Also, it has been proved in \cite{Chen} that the iterated suspension of $P^{n+1}(2)$ is a homotopy retract of $\tilde{L}_{2k+1}(P^{n+1}(2))$ (determined by the Lie elements ${\rm ad}^{k}([\mu,\nu](\mu))$ and ${\rm ad}^{k}([\mu,\nu](\nu))$). These two facts suggest that this summand may lie in $\tilde{Q}^{\rm max}_{2k+1}(P^{n+1}(2))$.
\end{remark}

\begin{remark}
In the next section, we will study a family of spaces which satisfy the condition $2)$ of our theorems. However, in general the study of hyperbolicity along our ideas heavily relies on the progress of the study of the functor $\tilde{Q}^{\rm max}$, for which the work in \cite{Wu2003} may be a first reference.
\end{remark}

\section{Possible candidates to be $p$-hyperbolic}
\noindent In this section, we may discuss some other examples which may be served as good candidates to have hyperbolic properties, and partly justify the choice of conditions of the theorems in previous section. As we pointed out before, one of the key points is to find some suspensions of the space under consideration which is the homotopy retract of the loop of the space. Further, it will be much helpful if the multiplicity of this suspension is greater than one. Our exposition in this section will exactly follow this idea.

Suppose $X$ is a suspended finite $CW$-complex localized at $p$ such that either $V_{{\rm odd}}=0$ or $V_{{\rm even}}=0$ where $V=\tilde{H}_\ast(X; \mathbb{Z}/p)$. Denote that ${\rm dim}~V=l$ and $M$ to be the sum of the degrees of the generators of $V$. Then in \cite{Beben}, Beben and the second author have proved that if $1<l<p-1$, $\Sigma^MX$ is a homotopy retract of $L_{l+1}(X)$. By their construction, $\Sigma^MX$ is homotopy equivalent to the telescope $T(g)$ of the composition map
\begin{equation*}
g: X^{\wedge(l+1)}\stackrel{f_{s_l}\wedge {\rm id}}{\longrightarrow}X^{\wedge(l+1)}\stackrel{f_{\beta_{l+1}}}{\longrightarrow}X^{\wedge(l+1)}.
\end{equation*}
For defining $f_{s_l}$ and $f_{\beta_{l+1}}$, first notice that we can define a right action of $\mathbb{Z}/p[S_k]$ on $V^{\otimes k}$ by permuting factors in a graded sense where $S_k$ denotes the symmetric group on $k$ letters. Then we may take three types of special elements in $\mathbb{Z}/p[S_k]$ to get the corresponding self-morphism of $V^{\otimes k}$:
\begin{equation*}
\hat{s}_k=\sum\limits_{\sigma\in S_k} {\rm sgn}~(\sigma) \sigma,
\end{equation*}
\begin{equation*}
\bar{s}_k=\sum\limits_{\sigma\in S_k}\sigma,
\end{equation*}
and the Dynkin-Specht-Wever elements $\beta_k$ (see Section $4$).
We now define $s_l$ by $s_l=\bar{s}_l$ if $V_{{\rm even}}=0$ and $s_l=\hat{s}_l$ if $V_{{\rm odd}}=0$. Since $X$ is a suspension, the self morphisms of $V^{\otimes (l+1)}$ determined by $s_l\otimes {\rm id}$ and $\beta_{l+1}$ can be realized as self-maps of $X^{\wedge(l+1)}$ which are exactly $f_{s_l}\wedge {\rm id}$ and $f_{\beta_{l+1}}$. Then $g$ can be formed as above such that ${\rm Im}~\tilde{H}_\ast(g)\cong \Sigma^MV$ and $\tilde{H}_\ast(g)$ is an idempotent up to an unit if $l<p-1$, and hence the telescope $T(g)$ is a homotopy retract of $X^{\wedge(l+1)}$. Further by checking the homology, $T(g)\simeq \Sigma^M X$ and indeed lies in $L_{l+1}(X)$.

By the above construction, we observe that there are other ways to produce $\Sigma^M X$ to be a homotopy retract of $L_{l+1}(X)$. That is, we simply let $s_l$ act on any others $l$ factors in $V^{\otimes(l+1)}$ rather than the first $l$ ones. To be precise, we may define the compositions
\begin{equation*}
g_{\sigma_j}: X^{\wedge(l+1)}\stackrel{\sigma_j}{\longrightarrow} X^{\wedge(l+1)}\stackrel{g}{\longrightarrow} X^{\wedge(l+1)}\stackrel{\sigma_j^{-1}}{\longrightarrow}X^{\wedge(l+1)},
\end{equation*}
for $\sigma_j=\bigl(\begin{smallmatrix}
    1 &  \cdots & j-1 & j      & j+1     & \cdots &  l+1\\
    1 & \cdots & j-1 & l+1      & j     & \cdots &  l
  \end{smallmatrix}\bigr)$
with $1\leq j\leq l+1$.  Notice that $g_{\sigma_{l+1}}=g$, and ${\rm Im}~\tilde{H}_\ast(g_{\sigma_j})={\rm Im}~\tilde{H}_\ast(g) \cdot\sigma_j$. Since we also have ${\rm Im}~\tilde{H}_\ast(g)= {\rm Im}~(s_l\otimes {\rm id})$ \cite{Beben}, we see that ${\rm Im}~\tilde{H}_\ast(g_{\sigma_j})$ is $\mathbb{Z}/p$-generated by $\big(s_l(x_1\otimes x_2\otimes\cdots \otimes x_l)\otimes x_i\big)\cdot \sigma_j$ with $1\leq i\leq l$ where $x_1$, $x_2$, $\ldots$, $x_l$ is a basis of $V$.

Our aim is to determine a lower bound of the multiplicity of $\Sigma^MX$ in the homotopy decomposition of $L_{l+1}(X)$. To achieve this, we first work on the homology level and prove that the sum of any $l$ out of the $l+1$ images ${\rm Im}~\tilde{H}_\ast(g_{\sigma_j})\in V^{\otimes (l+1)}$ ($1\leq j\leq l+1$) is indeed a direct sum, but the sum of all the $l+1$ images is not.
\begin{lemma}\label{linear}
The equality
\begin{equation}\label{dependent}
\sum\limits_{i,j} c_{i,j} \big(s_l(x_1\otimes x_2\otimes\cdots \otimes x_l)\otimes x_i\big)\cdot \sigma_j =0
\end{equation}
holds in $V^{\otimes(l+1)}$ for some constants $c_{i,j}$'s if and only if
\begin{equation}\label{coefficients}
c_{i,1}=-c_{i,2}=c_{i,3}=\cdots=(-1)^{l}c_{i,l+1}
\end{equation}
for any $1\leq i\leq l$.
\end{lemma}
\begin{proof}
First, we see that there are $l!$ monomials in the expansion of $s_l(x_1\otimes x_2\otimes\cdots \otimes x_l)\otimes x_i$, and then $l(l+1)l!=l(l+1)!$ monomials in the expansion of the left hand side of (\ref{dependent}), and the involved monomials are of the form
\begin{equation*}
x_{i_1}\otimes \cdots\otimes x_k\otimes \cdots \otimes x_k\cdots \otimes x_{i_{l+1}},
\end{equation*}
where $\{i_1,\ldots, k, \ldots, i_{l+1}\}=\{1,2, \ldots, l\}$, i.e. each $x_i$ appears at least once. We also notice that there are exactly $\binom{l+1}{2}l!=\frac{l(l+1)!}{2}$ distinct such monomials. On the other hand, any such monomial can be written as two different ways:
\begin{eqnarray*}
&&x_{i_1}\otimes \cdots\otimes x_k\otimes \cdots \otimes x_k\cdots \otimes x_{i_{l+1}}\\
&=& \big((x_1\otimes \cdots\otimes x_l)\cdot \tau_1\otimes x_k\big)\sigma_s\\
{\rm or} &=& \big((x_1\otimes \cdots\otimes x_l)\cdot \tau_2\otimes x_k\big)\sigma_t,
\end{eqnarray*}
where $s$ and $t$ refer to the two positions of $x_k$ in the monomial.
Hence each monomial appears at least twice in the expansion of the left hand side of (\ref{dependent}), and then appears exactly twice. These monomials are clearly linearly independent, and then (\ref{dependent}) holds if and only if $c_{k,s}+(-1)^{s-t-1}c_{k,t}=0$ for all $k$, $s$ and $t$.
\end{proof}

\begin{proposition}\label{multiBeben}
If $1<l<p-1$, the multiplicity of $\Sigma^M X$ in $L_{l+1}(X)$ is at least $l$, i.e., there exists a map
\begin{equation*}
\bigvee_{l}\Sigma^M X\hookrightarrow L_{l+1}(X)
\end{equation*}
which admits a left homotopy inverse.
\end{proposition}
\begin{proof}
We may prove the proposition by induction.
Let us consider the $l$ morphisms $g_{\sigma_j}$ with $1\leq j\leq l$. By Lemma \ref{linear}, we see that the sum of the homology images ${\rm Im}~\tilde{H}_\ast(g_{\sigma_j})\in V^{\otimes (l+1)}$ ($1\leq j\leq l$) is a direct sum ($c_{i,l+1}=0$ imples $c_{i,j}=0$). Suppose there exists at least $k$ copies of $\Sigma^MX$ in the homotopy decomposition of $X^{\wedge{l+1}}$, and further they are determined by $g_{\sigma_j}$ for $1\leq j\leq k$ respectively. Then we have the diagram of maps
\begin{equation*}
  \begin{tikzcd}
  &&\Sigma^MX\ar{d}{p_k\circ i_{k+1}}\ar[hook']{ld}[swap]{i_{k+1}}\\
  \bigvee_{j=1}^k\Sigma^MX \ar[hook]{r}{\iota_k:=\vee_j i_j} & X^{\wedge(l+1)}\ar{r}{p_k}&X^{\wedge(l+1)}/\bigvee_{j=1}^k\Sigma^MX=: Y_k  \ar{d}{p_{k,k+1}} \ar[bend left, dashed]{l}{s_{k}}
  \\
  &&Y_{k+1},\\
  \end{tikzcd}
\end{equation*}
 where $i_j: \Sigma^MX \simeq T(g_{\sigma_j})\rightarrow X^{\wedge(l+1)}$, $p_k\circ s_k\simeq {\rm id}$, and the row and the column are cofibre sequences. Consider the composition map $i_{k+1}^\prime:=s_k\circ p_k\circ i_{k+1}$. Since we know
 \begin{equation*}
 {\rm Im}~i_{k+1\ast} \cap \bigoplus_{j=1}^{k}~{\rm Im}~i_{j\ast}=0,
 \end{equation*}
then $p_{k\ast}\circ i_{k+1\ast}$ is injective which implies $i_{k+1\ast}^\prime=s_{k\ast}\circ p_{k\ast}\circ i_{k+1\ast}$ is injective.
Hence the composition map
\begin{equation*}
\Sigma^MX\vee X^{\wedge(l+1)}/\Sigma^MX\stackrel{ i_{k+1}^\prime \vee j_{k+1}}{\hookrightarrow} X^{\wedge(l+1)}\vee X^{\wedge(l+1)} \stackrel{\nabla}{\rightarrow}X^{\wedge(l+1)},
\end{equation*}
induces the isomorphism of $\mathbb{Z}/p$-homology, and then is a homotopy equivalence. Now there exists some map
\begin{equation*}
t_{k+1}^\prime: X^{\wedge(l+1)}\rightarrow \Sigma^MX
\end{equation*}
such that $t_{k+1}^\prime\circ  i_{k+1}^\prime\simeq ~{\rm id}$, i.e. $(t_{k+1}^\prime\circ s_{k})\circ( p_{k}\circ i_{k+1})\simeq {\rm id}$. It follows that the composition map
\begin{equation*}
Y_k\stackrel{{\rm comult}}{\longrightarrow} Y_k\vee Y_k\stackrel{t_{k+1}^\prime\circ s_{k}\vee p_{k,k+1}}{\longrightarrow} \Sigma^MX\vee Y_{k+1}
\end{equation*}
is a homotopy equivalence. Hence
\begin{equation*}
\Sigma^MX\vee\bigvee_{j=1}^k\Sigma^MX \hookrightarrow  Y_k\vee\bigvee_{j=1}^k\Sigma^MX\stackrel{\simeq}{\longrightarrow}  X^{\wedge(l+1)}
\end{equation*}
is a homotopy embedding which closes our inductive step, and we have proved that there are at least $l$ copies of $\Sigma^MX$ in the homotopy decomposition of $X^{\wedge(l+1)}$. By the proof of Proposition $4.2$ in \cite{Beben}, each $\Sigma^MX$ is indeed a homotopy retract of $L_{l+1}(X)$. Then the composition map
\begin{equation*}
\bigvee_{l}\Sigma^MX\rightarrow X^{\wedge(l+1)}\rightarrow L_{l+1}(X) \rightarrow X^{\wedge(l+1)} \rightarrow \bigvee_{l}\Sigma^MX
\end{equation*}
induces the isomorphism of $\mathbb{Z}/p$-homology, and hence $l$ copies of $\Sigma^MX$ indeed lie in $L_{l+1}(X)$.
\end{proof}

From this proposition, if we also know some suitable knowledge of the homotopy groups of the related spaces as that in Theorem \ref{hyper2} and others, then there is a good chance to prove that our space $X$ is locally hyperbolic.

\section{Appendix: The homotopy groups have at most exponential growth}
\noindent The upper bound of the size of the homotopy groups was determined by Henn \cite{Henn} using the unstable Adams spectral sequence (or Bousfield-Kan spectral sequence) where he considered the radius of convergence of the generating function of homotopy groups, and used a cohomological spectral sequence of Miller to study the $E_2$-page of BKSS. Here we reprove the result directly by using the lambda algebra description of the $E_1$-page. The treatment here also serves as a comparison to that in the proof of our previous results.
\begin{definition}\label{rlcs}
Given any group $G$ and a fixed prime number $p$, the \textit{lower $p$-central series} of $G$ is the filtration
\begin{equation*}
G=\Gamma_1 G\supseteq \Gamma_2 G\supseteq\ldots \supseteq\Gamma_s G\supseteq\ldots,
\end{equation*}
where
\begin{equation*}
\Gamma_s G=\{[x_1,x_2\ldots, x_l]^{p^t}~|~lp^t\geq s, x_i\in G\},
\end{equation*}
and $[x_1,x_2\ldots, x_l]=[[\ldots[x_1,x_2]\ldots], x_l]$ denotes the iterated commutator.
\end{definition}
If $G$ is a free group, then there is an isomorphism of restricted Lie algebras
\begin{equation*}
\bigoplus_s\Gamma_s G/\Gamma_{s+1} G \cong L(G/\Gamma_2 G)
\end{equation*}
between the direct sum of the quotients of subsequent terms in the lower $p$-central series and the free restricted Lie algebra over the abelianization of $G$, where for the restricted Lie structure of the direct sum, the bracket is the commutator operation and the $p$-th power morphism is derived from the power morphism in the group $G$.

Suppose $K$ is a simply connected simplicial set, we can form the Kan construction $GK$ such that $|GK|\simeq \Omega |K|$.
Then apply Definition \ref{rlcs} for each dimension of $GK$, we get a filtration of simplicial groups
\begin{equation*}
GK=\Gamma_1 GK\supseteq \Gamma_2 GK\supseteq\ldots \supseteq\Gamma_s GK\supseteq\ldots.
\end{equation*}
From its associated exact couple, we get a spectral sequence $\{E^{i}(K)\}_i$ which serves as the unstable version of Adams spectral sequence.
\begin{theorem}\cite{Rector}\label{SS}
Suppose $K$ is a simply connected simplicial set, then $\{E^{i}(K)\}_i$ converges to $E^\infty(K)$, and there are short exact sequences
\begin{equation*}
0\rightarrow F_{s+1}\pi_q(GK)\rightarrow F_{s}\pi_q(GK)\rightarrow E^\infty_{s, q}(K)\rightarrow 0,
\end{equation*}
where $\{F_{s}\pi_q(GK)={\rm Im}\big(^{(p)}\pi_q(\Gamma_s(GK))\rightarrow  ^{(p)}\pi_q(GK)\big)\}_s$ is the filtration of $\pi_q(GK)\cong \pi_{q+1}(K)$ modulo the subgroup of elements of finite order prime to $p$. Further, the $E_1$-term of the spectral sequence has the form
\begin{equation*}
E^1_{s,q}(K)=\pi_q(\Gamma_s GK/\Gamma_{s+1} GK) \cong\pi_q(L_s(GK/\Gamma_{2} GK)).
\end{equation*} \hfill $\Box$
\end{theorem}

The $E_1$-term of the spectral sequence can be described in terms of the so-called lambda algebra, the ${\rm mod}~2$ version of which is defined as follows:
\begin{definition}\cite{6A,BC}
For $p=2$, \textit{The lambda algebra $\Lambda$} is the graded associative differential algebra with unit over $\mathbb{Z}/2$ with a generator $\lambda_i$ of degree $i$ for each $i\geq 0$ and subject to the relations
\begin{equation*}
\lambda_i\lambda_{2i+1+k}=\sum\limits_{j\geq 0}\binom{k-1-j}{j}\lambda_{i+k-j}\lambda_{2i+1+j},
\end{equation*}
while the differential $\partial$ is given by
\begin{equation*}
\partial (\lambda_i)=\sum\limits_{j\geq 1}\binom{i-j}{j}\lambda_{i-j}\lambda_{j-1}.
\end{equation*}
\end{definition}

In \cite{6A}, it was proved that $\Lambda$ is isomorphic as differential algebra to $E^1S$ which is the $E^1$-term of the associated spectral sequence of the lower $2$-central series of $FS$, the free group spectrum of sphere spectrum (Section $4$ of \cite{KW65}). The additive structure of $\Lambda$ is similar to that of Steenrod algebra. We may denote $\lambda_I=\lambda_{i_1}\lambda_{i_2}\ldots\lambda_{i_s}$ with $I=(i_1,i_2\ldots, i_s)$ be any finite sequence of non-negative integers, and call $\lambda_I$ \textit{allowable} (or \textit{admissible}) if $2i_j\geq i_{j+1}$ for each $0<j<s$ or $I=\emptyset$. Then $\Lambda$ has a $\mathbb{Z}/2$-basis consisting of all the allowable monomials. Further, denote by $\Lambda(n)$ the additive subgroups of $\Lambda$ generated by the allowable monomials with some initial term $\lambda_i$ such that $i<n$. Then $\Lambda(n)$ is a differential sub-algebra of $\Lambda$ and in fact $\Lambda(n)\cong E^1(S^n)$.

\begin{theorem}\label{upper}
Suppose $K$ is a simply connected simplicial set such that $\tilde{H}_\ast(K;\mathbb{Z}/p)$ is finite dimensional. Then
the $p$-primary torsion part of $\pi_\ast(K)$ has at most exponential growth, i.e.,
\begin{equation*}
\limsup_n\frac{{\rm ln}~ T_n}{n}<\infty,
\end{equation*}
where $T_n$ is as in Definition \ref{hyperpdef}.
\end{theorem}
\begin{proof} First consider the case when $p=2$. In \cite{BC} and \cite{Curtis}, Bousfield and Curtis have proved that there exists a natural isomorphism
\begin{equation*}
\phi: L(\pi(V))\hat\otimes \Lambda \rightarrow \pi_\ast L(V)
\end{equation*}
for any simplicial vector space $V$ over $\mathbb{Z}/2$. Here $\hat\otimes$ is defined by
\begin{equation*}
M\hat\otimes \Lambda =\bigoplus_n M_n\otimes \Lambda(n)
\end{equation*}
for any graded vector space $M$. Hence $E^1$-term becomes
\begin{equation*}
E^1(K)\cong L(\pi_\ast(GK/\Gamma_2 GK))~\hat\otimes ~\Lambda\cong L(s^{-1}\tilde{H}_\ast K)~\hat\otimes~ \Lambda,
\end{equation*}
under which
\begin{equation*}
E^1_{2^sl,q}(K)\hookleftarrow\big(L_l(s^{-1}\tilde{H}_\ast K)~\hat\otimes~ \Lambda^s\big)_q,
\end{equation*}
and $\Lambda^s$ is the graded sub-vector space of $\Lambda$ generated by monomials of length $s$.

We will prove the theorem by estimating the size of the $E^1$-term, before which we have to deal with the terms involving $\lambda_0$.
By checking the definition of $\phi$ \cite{Curtis}, we see that for any $\tilde{x}\in \big(L_l(\pi_\ast(GK/\Gamma_2 GK))\big)_q$,
\begin{equation*}
\phi(\tilde{x}\otimes  \lambda_0)=\tilde{x}^2\in \pi_q(L_{2l}(GK/\Gamma_2 GK)).
\end{equation*}
Now if $x\in \pi_q(GK)$ is of order $2^r$, then there exists $l$ such that $x\in \Gamma_l(GK)$ and $x\not\in\Gamma_{l+1}(GK)$, and its image $\tilde{x}\in \pi_q\big(L_l(GK/\Gamma_2 GK)\big)\cong E^1_{l,q}(K)$ is a permanent cycle. Further, for each $1\leq i\leq r-1$, $x^{2^i}$ is presented by the permanent cycle $\tilde{x}^{2^i}\in E^1_{2^il,q}(K)$ which corresponds to $\tilde{x}\otimes \lambda_0^i$ under $\phi$. Their images $\tilde{x}^{2^i}_{\infty}$ ($0\leq i\leq r-1$) in the infinity page determine the element $x\in \pi_q(GK)$ through the filtration. On the other hand, if $\tilde{x}\otimes  \lambda_0$ is a permanent cycle in $E^1$-term, then it corresponds to some $x^2\in \pi_q(GK)$ which implies $x$ is also a Moore cycle. Hence in order to give an upper bound of the number of $2$-torsion summands in $\pi_\ast(GK)$, it suffices to give an upper bound of the dimension of $E_1$-term by omitting the elements of the form $\tilde{x}\otimes  \lambda_I\lambda_0$.

Now let us do the estimation. For the restricted Lie algebra $L(s^{-1}\tilde{H}_\ast K)$ which is contained in its universal enveloping algebra $T(s^{-1}\tilde{H}_\ast K)$, we have
\begin{equation*}
{\rm dim}~(L_{\leq l}(s^{-1}\tilde{H}_\ast K))\leq {\rm dim}~(T_{\leq l}(s^{-1}\tilde{H}_\ast K))\leq C_1 a^l,
\end{equation*}
where $a={\rm dim}~(\tilde{H}_\ast K)$ and $C_1$ is a positive constant.

For the lambda algebra, we may denote $\tilde{\Lambda}$ to be the graded sub-vector space of $\Lambda$ consisting of the allowable monomials ending with some $\lambda_i$ ($i>0$). Then we have
\begin{equation*}
{\rm dim}~(\tilde{\Lambda}_{\leq q})\leq \binom{2q}{q-1}\leq C_2 b^q
\end{equation*}
for some positive constant $C_2$ and $b$, where the first inequality can be deduced from that the dimension of $\tilde{\Lambda}_{\leq q}$ is smaller than the number of the non-negative integer solutions of the equation $x_1+x_2+\ldots+x_q+x_{q+1}=q$, and the second inequality can be obtained from the Stirling formula $n!\sim \sqrt{2\pi n}(\frac{n}{e})^n$.

Combining the previous arguments together, we have
\begin{equation*}
T_{q+1}\leq C_1 a^q \times C_2 b^q=C(ab)^q,
\end{equation*}
for $K$ is simply connected. Accordingly,
\begin{equation*}
\limsup_q\frac{{\rm ln}~ T_{q+1}}{q+1}\leq {\rm ln}~(ab)<\infty,
\end{equation*}
which prove the theorem for $p=2$.

Though a little more complicated, the similar argument can be applied to prove the theorem for any odd prime $p$.
\end{proof}

With the help of the stable version of the spectral sequence which from $E^2$-term on coincides with the Adams spectral sequence \cite{6A}, a similar argument will show the following theorem which should be known to experts:
\begin{theorem}\label{upperstable}
Given a simply connected finite spectrum $X$, the $p$-primary torsion of $\pi_\ast(X)$ has at most exponential growth, i.e.,
\begin{equation*}
\limsup_n\frac{{\rm ln}~ T_n}{n}<\infty. 
\end{equation*} \hfill $\Box$
\end{theorem}

\section*{Acknowledgements}
The authors would like to thank the anonymous referee most warmly for his/her valuable suggestions and comments on an earlier version of this paper, especially about the writing and presentation of our article which have essentially improved the exposition of this paper.


\begin{thebibliography}{10}

\bibitem{Anick} D. J. Anick, {\em Homotopy exponents for spaces of category two}, Algebraic Topology (Arcata, CA, 1986), Lecture Notes in Math. \textbf{1370} (Springer, Berlin, 1989), pp. 24-52.

\bibitem{Beben} P. Beben and J. Wu, {\em Modular representations and the homotopy of low rank $p$-local CW-complexes}, Math. Z. \textbf{273} (2013), 735-751.

\bibitem{BC} A. K. Bousfield and E. B. Curtis, {\em A spectral sequence for the homotopy of nice spaces}, Tran. Amer. Math. Soc. \textbf{151} (1970), 457-479.

\bibitem{6A} A. K. Bousfield, E. B. Curtis, D. M. Kan, D. G. Quillen, D. L. Rector and J. W. Schlesinger, {\em The ${\rm mod}~p$ lower central series and the Adams spectral sequence}, Topology \textbf{5} (1966), 331-342.

\bibitem{Chen} W. Chen and J. Wu, {\em Decomposition of loop spaces and periodic problem on $\pi_\ast$}, Algebr. Geom. Topl. \textbf{13} (2013), 3245-3260.

\bibitem{Cohen86} F. R. Cohen, {\em A course in some aspects of classic homotopy theory}, in: Lecture Notes in Math., \textbf{1286} (1986), Springer, Berlin, 1-92.

\bibitem{Cohen89} F. R. Cohen, {\em Applications of loop spaces to some problems in topology}, in: London Math. Soc. Lecture Notes, \textbf{139}, (1989), 10-20.

\bibitem{Cohen} F. R. Cohen, J. C. Moore and J. A. Neisendorfer, {\em Torsion in homotopy groups}, Ann. Math. \textbf{109} (1979), 121-168.

\bibitem{Cohen79} F. R. Cohen, J. C. Moore and J. A. Neisendorfer, {\em The double suspension and exponents of the homotopy group of spheres}, Ann. Math. \textbf{109} (1979), 549-565.

\bibitem{Cohen1} F. R. Cohen, J. C. Moore and J. A. Neisendorfer, {\em Exponent in homotopy theory}, from: Algebraic topology and algebraic $K$-theory (Princeton N. J., 1983), (W. Browder, editor), Ann. Math. Stud. \textbf{113}, Princeton Univ. Press (1987) 3-34.

\bibitem{Cohen95} F. R. Cohen and J. Wu, {\em A remark on the homotopy groups of $\Sigma^n\mathbb{R}P^2$}, from: The \v{C}ech centennial, (M. Cenkl, H. Miller, editors), Contemp. Math. \textbf{181}, Amer. Math. Soc. (1995) 65šC81.

\bibitem{Curtis} E. B. Curtis, {\em Simplicial homotopy theory}, Advances in Math. \textbf{6} (1971), 107-209.

\bibitem{GTM205} Y. F\'{e}lix, S. Halperin and J.- C. Thomas, {\em Rational homotopy theory}, Graduate Texts
in Mathematics, Vol. \textbf{205}, Springer-Verlag, New York, 2001.

\bibitem{Felix} Y. F\'{e}lix, S. Halperin and J.- C. Thomas, {\em Exponential growth and an asymptotic formula for the ranks of homotopy groups of a finite $1$-connected complex}, Ann. Math. \textbf{170} (2009), 443-464.

\bibitem{Henn} H. -W. Henn, {\em On the growth of homotopy groups}, Manuscripta Math. \textbf{56} (1986), 235-245.

\bibitem{Hill} M. A. Hill, M. J. Hopkins and D. C. Ravenel, {\em On the nonexistence of elements of Kervaire invariant one}, Ann. Math. \textbf{184} (2016), 1-262.

\bibitem{KW65} D. M. Kan and G. W. Whitehead, {\em The reduced join of two spectra}, Topology \textbf{3} (1965), 239-261.

\bibitem{MN} C. A. McGibbon and J. A. Neisendorfer, {\em On the homotopy groups of a finite dimensional space}, Comm. Math. Helv. \textbf{59} (1984), 253-257.

\bibitem{MW} C.A. McGibbon and C.W. Wilkerson, {\em Loop spaces of finite complexes at large primes}, Proc. Amer. Math. Soc. \textbf{96} No. 4 (1986), 698-702.

\bibitem{Neisen3} J. A. Neisendorfer, {\em $3$-primary exponents}, Math. Proc. Camb. Phil. Soc. \textbf{90} (1981), 63-83.

\bibitem{Neisen83} J. A. Neisendorfer, {\em The exponent of a Moore space}, from: Algebraic topology and algebraic $K$-theory (Princeton N. J., 1983), (W. Browder, editor), Ann. Math. Stud. \textbf{113}, Princeton Univ. Press (1987) 35-71.

\bibitem{Neisen} J. A. Neisendorfer, {\em Algebraic methods in unstable homotopy theory}, Cambridge Univ. Press (2009).

\bibitem{Neisen2013} J. A. Neisendorfer, {\em Samelson products and exponents of homotopy groups}, J. Homo. Rela. Stru., \textbf{8} (2) (2013), 239-277.

\bibitem{Rector} D. L. Rector, {\em An unstable Adams spectral sequence}, Topology \textbf{5} (1966), 343-346.

\bibitem{Selick83} P. S. Selick, {\em On conjectures of Moore and Serre in the case of torsion-free suspensions}, Math. Proc. Camb. Phil. Soc. \textbf{94} (1983), 53-60.

\bibitem{Selick} P. S. Selick and J. Wu, {\em On natural coalgebra decompositions of tensor algebras and loop susupensions}, Memoirs AMS \textbf{148} (2000), No. 701.

\bibitem{Selick01} P. S. Selick and J. Wu, {\em On functorial decompositions of self-smash products}, Manuscripta Mathematica \textbf{111} (4) (2003), 435-457.

\bibitem{Selick1} P. S. Selick and J. Wu, {\em The functor $A^{{\rm min}}$ on $p$-local spaces}, Math. Z. \textbf{253} (2006), 435-451.

\bibitem{Serre} J. P. Serre, {\em Cohomologie modulo $2$ des complexes d'Eilenberg-MacLane}, Comm. Math. Helv. \textbf{27} (1953), 198-232.

\bibitem{Stanley} D. Stanley, {\em Exponents and suspension}, Math. Proc. Cambridge Philos. Soc. \textbf{133} (1) (2002), 109-116.

\bibitem{Theriault} S. D. Theriault, {\em Homotopy exponents of ${\rm mod}~2^r$ Moore spaces}, Topology \textbf{47} (2008), 369-398.

\bibitem{Umeda} Y. Umeda, {\em A remark on a theorem of J.-P. Serre}, Proc. Japan Acad. \textbf{35} (1959), 563-566.

\bibitem{White} G. Whitehead, {\em Elements of homotopy theory}, Springer-Verlag, GTM \textbf{62}, (1978).

\bibitem{Wu} J. Wu, {\em On combinatorial calculations for the James-Hopf maps}, Topology \textbf{37} (5) (1998), 1011-1023.

\bibitem{Wu2003} J. Wu, {\em Homotopy theory and the suspensions of the projective plane}, Memoirs AMS \textbf{162} (2003), No. 769.

\end{thebibliography}
\end{document}